\documentclass{article}
\usepackage[utf8]{inputenc}
\usepackage{amsmath}
\usepackage{amssymb}
\usepackage{amsthm}
\usepackage{graphicx}
\usepackage{authblk}
\usepackage{hyperref}

\usepackage{euler}
\DeclareFontShape{OT1}{cmr}{m}{up}
      {<->ssub*cmr/m/n}{}
\DeclareFontShape{OT1}{cmr}{b}{up}
      {<->ssub*cmr/b/n}{}

\usepackage[inline]{enumitem}

\usepackage{mathtools}
\mathtoolsset{ 
    showonlyrefs,
    showmanualtags
}
  
\numberwithin{equation}{section}

\newtheorem{Theorem}{Theorem}[section]
\newtheorem{Lemma}[Theorem]{Lemma}
\theoremstyle{definition}
\newtheorem*{Definition}{Definition}

\usepackage{geometry}  

\newcommand{\term}[1]{\textit{#1}}

\DeclareMathOperator{\deSitter}{d\mathbb{S}^3}
\newcommand{\RR}{\mathbb{R}}
\newcommand{\HH}{\mathbb{H}}
\renewcommand{\SS}{\mathbb{S}}

\DeclareMathOperator{\inv}{Inv}
\DeclareMathOperator{\conv}{conv}
\DeclareMathOperator{\poly}{poly}
\DeclareMathOperator{\trunc}{trunc}
\DeclareMathOperator{\ee}{e}
\newcommand{\ip}[2]{\langle#1, #2\rangle}

\title{Rigidity of Koebe Polyhedra \\ and Inversive Distance Circle Packings}

\author[1]{John C.~Bowers}
\author[2]{Philip L.~Bowers}
\author[3]{Carl O.~R.~Lutz}
\affil[1]{Department of Computer Science, James Madison University, Harrisonburg, VA 22807, USA, \href{mailto:bowersjc@jmu.edu}{bowersjc@jmu.edu}
} 
\affil[2]{Department of Mathematics, The Florida State University, Tallahassee, FL 32306, USA,  \href{mailto:pbowers@fsu.edu}{pbowers@fsu.edu}
} 
\affil[3]{Department of Mathematics, University of Luxembourg,
Maison du nombre, 6 avenue de la Fonte, L-4364 Esch-sur-Alzette, Luxembourg,
\href{mailto:clutz@math.tu-berlin.de}{carl.lutz.math@gmail.com}
}

\begin{document}

\maketitle

\begin{abstract}\noindent
Hyperbolic inversive distance circle packings on the $2$-sphere correspond to Koebe polyhedra in the Beltrami--Klein model $\mathbb{B}^{3}$ of hyperbolic $3$-space. Koebe polyhedra are triangulated convex hyperbolic polyhedra with hyperideal vertices whose faces meet $\mathbb{B}^{3}$. We prove the global rigidity of these circle packings or, equivalently, of these Koebe polyhedra under mild assumptions on the links of their vertices. Previous rigidity results apply only when all edges of the Koebe polyhedron are tangent or, alternatively, when no edge is tangent to the ideal boundary of hyperbolic space. We remove these restrictions. This generalizes the global rigidity results of both Bao--Bonahon \cite{BaoBonahon:2002} and Bowers--Bowers--Pratt \cite{BBP18}, as well as the uniqueness part of the celebrated Koebe--Andre'ev--Thurston Theorem to the case where adjacent circles need not touch~\cite{pK36,Andreev:1970a,Andreev:1970b,Thurston:1980,B19}.
\end{abstract}

\section{Introduction} 

\paragraph{History and context.}
Paul Koebe \cite{pK36} proved his circle packing theorem in 1936 as an application of his famous uniformization theorem that avers that each finitely connected planar domain is conformally homeomorphic with a circle domain, a domain in the plane all of whose boundary components are round Euclidean circles. He showed that a simplicial triangulation $K$ of the $2$-sphere $\mathbb{S}^{2}$ corresponds to a circle packing $\mathscr{C}(K) = \{C_{v} = \partial D_{v}: v\in V=K^{(0)}\}$, a collection of circles bounding disks in $\mathbb{S}^{2}$ indexed by the vertex set $V$ of $K$ whose disks have pairwise disjoint interiors and with $C_{u}$ tangent to $C_{v}$ whenever $uv$ is an edge of $K$. Each circle packing $\mathscr{C}(K)$ determines a convex triangulated polyhedron $\mathscr{KP}(K)$ in the real projective $3$-space $\mathbb{RP}^{3}$ whose vertex set is the set of conical cone points determined by $K$. The \term{conical cone point} determined by a disk $D$ on $\mathbb{S}^{2}$ is the vertex of the unique right circular cone whose intersection with $\mathbb{S}^{2}$ is $\partial D$. This \term{Koebe polyhedron} $\mathscr{KP}(K)$ has the following properties: it is convex and triangulated, each edge is tangent to the $2$-sphere, and each triangular face meets the open $3$-ball $\mathbb{B}^{3}$. In fact, identifying the vertex set $V$ with the conical cone points of $\mathscr{KP}(K)$, the intersection of the triangular face $uvw$ with $\mathbb{S}^{2}$ is a circle orthogonal to the three circles $C_{u}$, $C_{v}$, and $C_{w}$, this the \term{orthocircle} determined by the face $uvw$. Notice that the edge $uv$ of the Koebe polyhedron is tangent to the $2$-sphere at the point of intersection of the circle $C_{u}$ and $C_{v}$ and the orthocircle of face $uvw$ passes through the tangency points of its corresponding three circles. 

\begin{figure}[t!]
	\centering
	\includegraphics[width=0.5\textwidth]{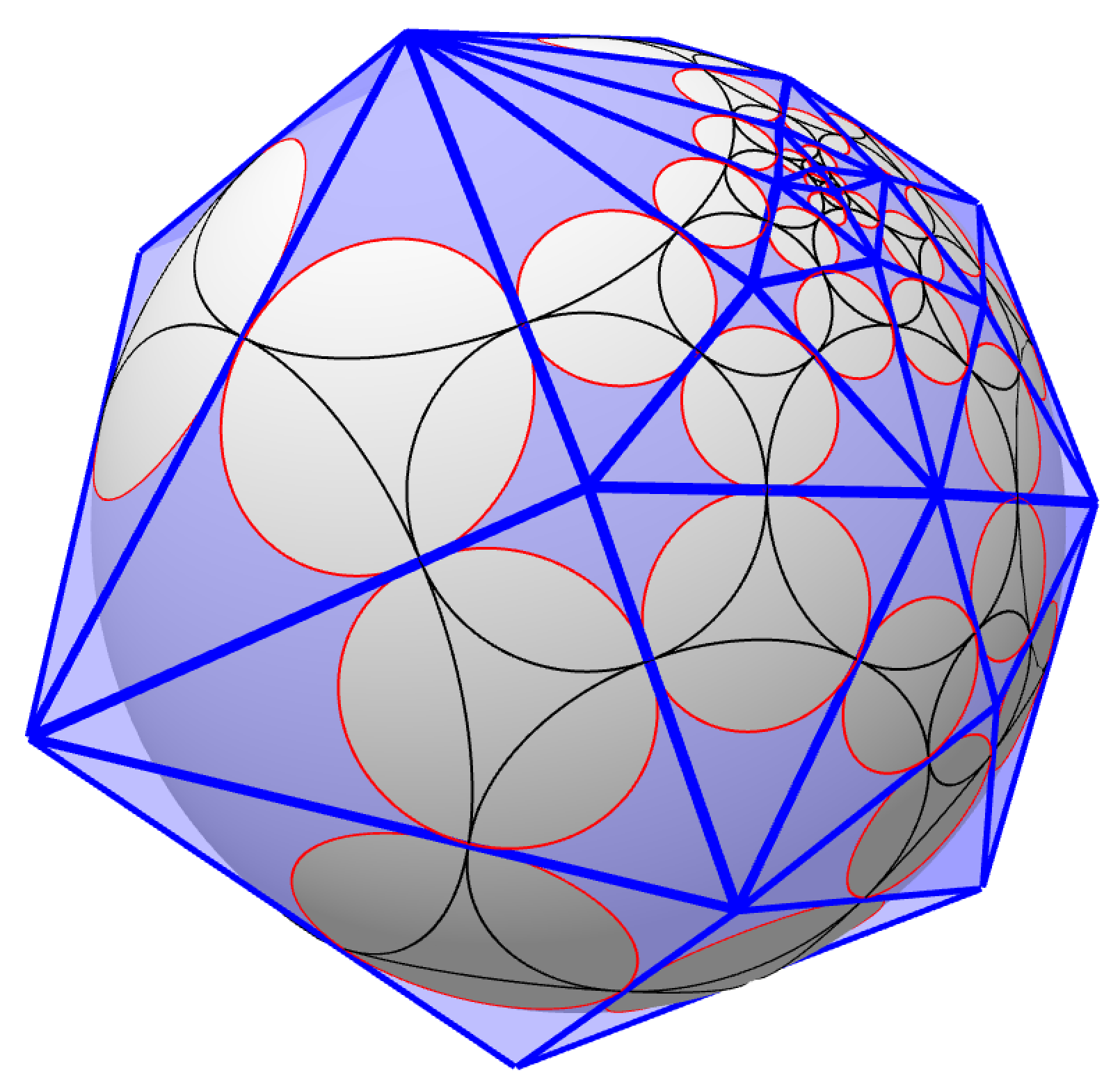}
	\caption{A circle packing and its associated Koebe polyhedron.
	    The black circles are the circles of the packing and the red ones are the orthocircles formed by the intersections of the blue triangular faces of the Koebe polyhedron with the $2$-sphere.}
	\label{fig:KoebePoly}
\end{figure}

Koebe proved that his circle packing is globally rigid with respect to the M\"obius group acting on the $2$-sphere: any two packings with the combinatorics of $K$ are congruent by a M\"obius transformation of $\mathbb{S}^{2}$. Since the M\"obius action on $\mathbb{S}^{2}$ extends to a projective action on the real projective $3$-space $\mathbb{RP}^{3}$ that set-wise fixes $\mathbb{S}^{2}$, this guarantees the global rigidity of the corresponding Koebe polyhedron. Identifying $\mathbb{B}^{3}$ as the Beltrami-Klein model of hyperbolic $3$-space, the intersection of the Koebe polyhedron with $\mathbb{B}^{3}$ is said to be a \term{hyperideal hyperbolic polyhedron}, with \term{hyperideal} meaning that its vertices lie beyond infinity. Since the M\"obius action on $\mathbb{B}^{3}$ is by isometries in the Beltrami-Klein hyperbolic metric, the Koebe rigidity of circle packings implies the hyperbolic rigidity of these hyperideal polyhedra.

Our interest is in the global rigidity of various generalizations of circle packings and their corresponding Koebe polyhedra. There is a rich literature from the last 30 years in which circle packings have been generalized in various ways: by allowing adjacent circles to overlap by specified angles (shallow overlaps and deep overlaps) \cite{Thurston:1980}, by extending the theory to higher genus surfaces \cite{BSt90,Thurston:1980}, by allowing adjacent circles to be disjoint by specified inversive distances \cite{BS04}, by allowing companion disks that are not adjacent to overlap (branch structures) \cite{BS96}, by considering shapes other than circles to pack \cite{Schr96}, and by relaxing the requirement that $K$ triangulate (circle polyhedra) \cite{BBP18}. In each case of generalization, questions of existence and uniqueness (global and infinitesimal rigidity) have been addressed and answered, but there are remaining open questions. In this paper, we address one of the remaining open cases of global rigidity.

\paragraph{The questions addressed.}
Bao and Bonahon \cite{BaoBonahon:2002} studied the existence and rigidity of hyperideal polyhedra in $\mathbb{B}^{3}$ and proved the global rigidity of those,  all of whose edges meet $\mathbb{B}^{3}$. By considering the dual polyhedron, this also verified the global rigidity of those, all of whose edges lie beyond infinity but all of whose faces still meet $\mathbb{B}^{3}$. The work of \cite{BBP18} generalized the rigidity results of \cite{BaoBonahon:2002} by allowing some edges to meet, and others to lie beyond $\mathbb{B}^{3}$. Specifically, it verified the rigidity of those hyperideal polyhedra all of whose faces meet $\mathbb{B}^{3}$ but none of whose edges lie tangent to $\mathbb{B}^{3}$. This left open the question of global rigidity of hyperideal polyhedra when individual edges have no constraints and are allowed to meet, to lie tangent to, or to lie beyond $\mathbb{B}^{3}$. The present article verifies global rigidity in this general setting whenever the polyhedron is triangulated and strictly convex.

\begin{Theorem}\label{TheoremA}
    Let $\mathscr{P}$ be a proper, strictly convex and triangulated hyperideal polyhedron, all of whose faces meet $\mathbb{B}^{3}$. Then $\mathscr{P}$ is globally rigid in hyperbolic $3$-space.
\end{Theorem}

\noindent Any hyperideal polyhedron $\mathscr{P}$ of the theorem determines a corresponding hyperbolic inversive distance circle packing. These are circle packings with specified inversive distances for each pair of adjacent circles, and for which each face of circles has an orthocircle, this the \term{hyperbolic} condition. Letting $K$ be the simplicial complex corresponding to $\mathscr{P}$, the polyhedron $\mathscr{P}=\mathscr{KP}(K)$ is the \term{Koebe polyhedron} determined by the circle packing. The Koebe polyhedron is convex and triangulated and each triangular face meets the open $3$-ball $\mathbb{B}^{3}$, but the edges need not be tangent to $\mathbb{B}^{3}$ as is the case when adjacent circles are tangent. Theorem~\ref{TheoremA} verifies the M\"obius rigidity of Koebe polyhedra and translates to the following rigidity for circle packings.

\begin{Theorem}\label{TheoremB}
    Let $\mathscr{C}(K)$ be a proper, strictly convex hyperbolic inversive distance circle packing of the sphere $\mathbb{S}^{2}$, where $K$ triangulates $\mathbb{S}^{2}$. Then $\mathscr{C}(K)$ is unique up to M\"obius transformations.
\end{Theorem}

\noindent This latter theorem generalizes the rigidity of the Koebe--Andre'ev--Thurston Theorem (KAT) to more general inversive distance circle packings in which adjacent circles may, variously, overlap by a specified angle, be tangent, or be separated by a specified inversive distance.

The condition of \term{properness} in these two theorems was introduced in \cite{BBP18} and guarantees that the link of a circle in a packing, or the link of a hyperideal vertex in a polyhedron, behaves nicely. In \cite{BBP18}, inversive distance circle packings were generalized to \term{circle polyhedra} -- collections of circles in the pattern of polyhedra, not necessarily triangulated -- and properness of links of vertices was introduced to rid the theory of pathological cases that are rather scarce. Properness for circle polyhedra behaves somewhat analogously to boundedness for Euclidean polytopes, and has been a necessary component in prior work. Conditions that guarantee properness have not been explored previously. Our analysis in this paper shows that properness follows when there are no \term{deep overlaps} -- when no pair of circles overlaps by more than an angle of $\pi/2$. This includes all the cases that appear in KAT, \cite{BaoBonahon:2002}, and \cite{BBP18}, and in most of the other literature on circle packings. Nonetheless, the converse is not true: there exist proper convex circle polyhedra with overlap angles greater than $\pi/2$, and our results cover these cases as well.

\paragraph{Outline of the paper.}
A feature of this paper is that we work in Minkowski space $\mathbb{R}^{3,1}$ in order to linearize some of our arguments. In this view, circle polyhedra become polyhedra in the de Sitter sphere. We then are able to import recent work on the local rigidity of circle polyhedra from \cite{bowers2020}, where convex circle polyehdra are shown to be infinitesimally rigid using techniques from the classical rigidity theory of bar and joint frameworks. This interaction between classical rigidity theory and circle configurations on the sphere has been an ongoing theme in recent research; see, for instance, \cite{Ma:2012hl, bowersBowers:2016,BBP18, BBP19, bowers2020, CG19}. 

In Section~\ref{Prelims} we review the relationship between the Lorentz geometry of Minkowski space $\mathbb{R}^{3,1}$ and the M\"obius geometry of the $2$-sphere boundary $\mathbb{S}^{2}$ of the Beltrami--Klein model $\mathbb{B}^{3}$ of hyperbolic space. We parameterize the round disks on $\mathbb{S}^{2}$, and what is the same, the hyperbolic half-spaces in $\mathbb{B}^{3}$, by the points of the de Sitter sphere $\deSitter$, the sphere of radius $+1$ in $\mathbb{R}^{3,1}$. Circle polyhedra in $\mathbb{S}^{2}$ then correspond to polyhedra in $\mathbb{R}^{3,1}$ whose vertices lie on $\deSitter$. After reviewing the inversive distance between disks in terms of the Lorentz inner product, we define the configuration space of a triangulated circle polyhedron and its corresponding inversive measure space. 

Properness for circle polyhedra was defined in \cite{BBP18} in case no two circles are tangent. In Section~\ref{ProperDefinition} we extend the definition of properness to cover the cases where circles are allowed to be tangent. We also report conditions on the configuration of circles in a circle polyhedron that guarantee properness.

The natural mapping from the configuration space to the measure space is analyzed in Section~\ref{UniqueTriangulated} using techniques from classical rigidity theory, which will allow us to prove that proper strictly convex triangulated circle polyhedra are globally rigid. The main ideas arose in the context of the first author's work in \cite{bowers2020} where techniques from classical rigidity theory informed the arguments. 

In Section~\ref{ProperProofs} properness for circle polyhedra is revisited and examined in detail. There we refine the definition of properness and then derive an equivalent formulation that uses the pencils of circles determined by the adjacent vertex pairs of the circle polyhedron. This allows us to give a quick proof that the conditions reported in Section~\ref{ProperDefinition} do indeed guarantee properness.

\section{Preliminaries}\label{Prelims}

\subsection{Correspondence between disks on \texorpdfstring{$\mathbb{S}^2$}{S²} and points on the de Sitter sphere and the inversive distance between disks}

We make heavy use of the correspondence between disks on the sphere $\mathbb{S}^2$ and points on the unit de Sitter sphere in the Minkowski 4-space $\mathbb{R}^{3,1}$. We refer the reader to \cite{bowers2020} for a more thorough treatment, but recall here the details necessary for this paper. 

The Minkowski 4-space $\mathbb{R}^{3,1}$ is obtained by replacing the usual Euclidean inner product in $\mathbb{R}^4$ with the Lorentzian inner product. A point of $\mathbb{R}^{3,1}$ is given by a 4-tuple $(x, y, z, t)$. The first three coordinates are typically called the space coordinates and the last is called the time coordinate. Let ${\bf u}_1 = (a_1, b_1, c_1, d_1)$ and ${\bf u}_2 = (a_2, b_2, c_2, d_2)$ be two vectors in $\mathbb{R}^{3,1}$. The \term{Lorentzian inner product} is given by $\langle {\bf u}_1, {\bf u}_2 \rangle_{3,1} = a_1 a_2 + b_1 b_2 + c_1 c_2 - d_1 d_2$.
The \term{Lorentz norm} of a vector ${\bf u}$ is defined as $\|{\bf u}\|_{3,1}:=\sqrt{\langle {\bf u}, {\bf u}\rangle_{3,1}}$. Given a non-zero vector ${\bf n} = (a, b, c, d)$, the set $\Pi_{\bf n} = \{{\bf u}\in\mathbb{R}^{3,1} : \langle {\bf n}, {\bf u}\rangle_{3,1} = 0\}$ is the hyperplane through the origin of $\mathbb{R}^{3,1}$ defined by the equation $ax + by + cz - dt = 0$. The vector $\mathbf{n}$ that defines the hyperplane $\Pi_{\mathbf{n}}$ is called the \term{Lorentz normal} to $\Pi_{\mathbf{n}}$. The \term{(unit) de Sitter sphere} is the set $\deSitter =\{{\bf u}\in\mathbb{R}^{3,1}:\langle {\bf u}, {\bf u}\rangle_{3,1} = 1\}$. The \term{light cone} is the set $\{{\bf u}\in\mathbb{R}^{3,1}:\langle {\bf u}, {\bf u}\rangle_{3,1} = 0\}$.

The $t=1$ subspace of $\mathbb{R}^{3,1}$ serves as a model of Euclidean 3-space $\mathbb{E}^3$. The intersection of the light cone with $\mathbb{E}^3$ is the unit sphere $\mathbb{S}^2$. A \term{spacelike} vector ${\bf u}$ is one where $\langle {\bf u}, {\bf u}\rangle_{3,1} > 0$. For the spacelike vector ${\bf n} = (a, b, c, d)$, the hyperplane $\Pi_{\bf n}$ with Lorentz normal ${\bf n}$ intersects the interior of $\mathbb{S}^2$ and meets $\mathbb{S}^2$ at the circle $$C_{\bf n} = \{(x, y, z, 1)\in\mathbb{R}^{3,1}: ax + by + cz - d = 0\ \ \mathrm{and}\ \ x^2 + y^2 + z^2 = 1\}.$$ $C_{\bf n}$ bounds two disks on $\mathbb{S}^2$: the positively oriented $$D_{\bf n}^+ = \{(x, y, z, 1)\in\mathbb{R}^{3,1}: ax + by + cz - d > 0\ \ \mathrm{and}\ \ x^2 + y^2 + z^2 = 1\},$$ and the negatively oriented $$D_{\bf n}^- = \{(x, y, z, 1)\in\mathbb{R}^{3,1}: ax + by + cz - d < 0\ \ \mathrm{and}\ \ x^2 + y^2 + z^2 = 1\}.$$ Replacing ${\bf n}$ with $-{\bf n}$ yields the same circle (i.e. $C_{\bf n} = C_{(-{\bf n})}$), but flips the orientation so that $D_{\bf n}^+ = D_{(-{\bf n})}^-$ and $D_{\bf n}^- = D_{(-{\bf n})}^+$. 

The map ${\bf n}\mapsto D_{\bf n}^+$ defined for all spacelike ${\bf n}\in \mathbb{R}^{3,1}$ gives a bijection between oriented disks on the sphere $\mathbb{S}^2$ and spacelike rays in $\mathbb{R}^{3,1}$ (a ray here is the equivalence class of a vector ${\bf n}$ under positive scaling). Each ray pierces the de Sitter sphere at exactly one point, and so we also obtain a bijection between oriented disks on the sphere $\mathbb{S}^2$ and points on $\deSitter$. In the remainder we will use these bijections interchangeably, considering a point on the de Sitter sphere, its corresponding spacelike ray, and its corresponding disk on $\mathbb{S}^2$ as equivalent. For example, when the points $\mathbf{n}_{i}$ on the de Sitter sphere correspond to the disks $D_{i}$, for $i=1,2$, we write $\langle D_1, D_2\rangle_{3, 1}$ for $\langle \mathbf{n}_1, \mathbf{n}_2\rangle_{3, 1}$. Furthermore, we will drop the $(+)$-superscript for a positively oriented disk and use the unadorned $D$ to denote the positively oriented disk corresponding to $\mathbf{n} =(a, b, c, d)\in \deSitter$, and use $D^{-}$ for its negatively oriented cousin. We say a set of disks is variously linearly independent or dependent precisely when their corresponding rays are. 

\paragraph{The inversive distance.}
Given two disks $D_1$ and $D_2$ represented as two spacelike vectors, the \term{inversive distance} between $D_1$ and $D_2$ is $$\mathrm{Inv}(D_1,D_2) = -\frac{\langle D_1, D_2\rangle_{3,1}}{\|D_1\|_{3,1} \|D_2\|_{3,1}}.$$ \noindent If the vectors have unit length (i.e., have their endpoints on the de Sitter sphere), then the inversive distance simplifies to 
\begin{equation}\label{eq:inversivedistance}
    \mathrm{Inv}(D_1, D_2) = -\langle D_1, D_2\rangle_{3, 1}.
\end{equation}

\subsection{The restricted Lorentz and M\"obius groups}

The orientation preserving isometries of $\mathbb{R}^{3,1}$ (Lorentz inner product preserving maps) that fix the origin and preserve the time direction form a group called the \term{restricted Lorentz group} $\mathrm{SO}^+(3, 1)$. This group is isomorphic to the 6-dimensional M\"obius group on $\mathbb{S}^2$. The elements of this group map the light cone to itself. The action of this group restricted to $\mathbb{S}^2$ is the M\"obius group. In the remainder, when we refer to a M\"obius transformation, we mean simultaneously its action on $\mathbb{S}^2$ and its extension to $\mathbb{R}^{3, 1}$ as a Lorentz transformation. The M\"obius group is not compact, which we shall see requires work from us in the proof of our main theorem. 

Given two triples of disks, $D_1, D_2, D_3$ and $D_1', D_2', D_3'$, there exists a M\"obius transformation mapping $D_i\mapsto D_i'$ if and only if $\langle D_i, D_j\rangle_{3, 1}=\langle D_i', D_j'\rangle_{3, 1}$ for each pair $ij$. This transformation is unique provided that $D_1, D_2, D_3$ corresponds to linearly independent points on the de Sitter sphere. 

\subsection{Circle polyhedra} 

A polyhedron in the de Sitter sphere is the projection of a polyhedral cone with the origin of $\mathbb{R}^{3, 1}$ as its apex. All of the rays from the origin to the vertices of the polyhedron are spacelike. Since the linear subspaces of $\mathbb{R}^{3,1}$ and $\mathbb{R}^4$ are the same, this cone is in convex position precisely when it is convex in the usual sense in $\mathbb{R}^4$ -- given any three rays in the cone, the hyperplane supporting the rays does not separate the remaining rays. It is further strictly convex when the the hyperplane supporting the rays does not contain any other ray. In particular, this means that any four rays corresponding to vertices of the polyhedron are linearly independent. 

The rays of this cone pierce the de Sitter sphere at the vertices of the de Sitter polyhedron, with the convex combination of two rays corresponding to an edge, and a face corresponding to all convex combinations of its vertices. The 1-skeleton of this polyhedron is a 3-connected planar graph $G = (V(G), E(G), F(G))$. Each vertex of a de Sitter polyhedron corresponds to a disk on $\mathbb{S}^2$, and hence the entire de Sitter polyhedron has a dual pattern of disks on $\mathbb{S}^2$ called its \term{circle polyhedron} (or c-polyhedron for short).  

Synthetically, we can define a circle polyhedron by giving a pair $(G, D)$ of a 3-connected planar graph $G$ (called its combinatorics) together with a set of \term{vertex-disks} $D = \{D_i\subset\mathbb{S}^2 : i\in V(G)\}$ on $\mathbb{S}^2$.  When the dual de Sitter polyhedron is (strictly) convex, we say the c-polyhedron is as well.

\paragraph{Unitary c-polyhedra.}
Given a c-polyhedron $P=(G, D)$, if there is an edge $ij$ such that the inversive distance between $D_i$ and $D_j$ is $1$, and hence the disks are tangent, we say $P$ is \term{unitary}, otherwise \term{non-unitary}. 

\paragraph{Hyperbolic c-polyhedra.}
A hyperplane in $\mathbb{R}^{3, 1}$ may intersect the interior of the light cone, be tangent to it, or meet it only at the origin. If all the support hyperplanes of a de Sitter polyhedron intersect the interior of the light cone, we call the polyhedron (and its corresponding circle polyhedron) \term{hyperbolic}. 

A face of a hyperbolic c-polyhedron has a beautiful geometric property. Each of its vertex-disks is orthogonal to a single circle, which we call its \term{orthocircle}, which is precisely the intersection of $\mathbb{S}^2$ with the support hyperplane of the face in $\mathbb{R}^{3,1}$. The convexity of a hyperbolic c-polyhedron is equivalent to the following property on $\mathbb{S}^2$ -- for each face of the c-polyhedron, one of the disks bounded by the orthocircle of the face has inversive distance $\ge 0$ to each of the disks not in the face (this is strict convexity when these inversive distances are all $>0$). 

\paragraph{Congruencies of c-polyhedra.}
Let $P = (G, D)$ and $P' = (G, D')$ be two c-polyhedra with the same combinatorics $G$. Consider a face $f$ of $G$ with vertices $v_1, \dots, v_k$ and let $D_1, \dots D_k$ be the disks of $f$ in $P$ and $D_1', \dots, D_k'$ be the disks of $f$ in $P'$. If there is a single M\"obius transformation that maps each $D_i$ to $D_i'$ then we say that the faces of $P$ and $P'$ corresponding to $f$ are \term{(M\"obius) congruent}. If each of the corresponding faces of $P$ and $P'$ is congruent, we say that $P$ and $P'$ are \term{locally (M\"obius) congruent}. Note that if $P$ and $P'$ are locally congruent, then for each edge $ij\in E$ it must be the case that the inversive distance between $D_i$ and $D_j$ equals the inversive distance between $D_i'$ and $D_j'$. If there is a single M\"obius transformation taking each disk $D_i$ of $P$ to $D_i'$ of $P'$, then we say that $P$ and $P'$ are \term{globally (M\"obius) congruent}. We now state the main theorem of \cite{BBP18} in our terms. 

\begin{Theorem}[\cite{BBP18}]\label{thm:bbp18}
Let $P$ and $Q$ be two proper convex hyperbolic \textbf{non-unitary} circle polyhedra. If $P$ and $Q$ are locally congruent, then $P$ and $Q$ are globally congruent. 
\end{Theorem}
\noindent This, without the non-unitary assumption, is the circle polyhedron version of the famous Cauchy Rigidity Theorem for convex polyhedra in $\mathbb{R}^{3}$. The proof of the theorem in \cite{BBP18} required the irritating non-unitary assumption, and the primary goal of this article is to remove that assumption.

\subsection{The configuration and measure spaces of triangulated c-polyhedra}
\label{sec:configurationspace}

Let $G$ be a triangulated 3-connected planar graph with $n$ vertices and $m=3n-6$ edges and denote the vertices of $G$ by $1, \dots, n$. The \term{configuration space} of $G$ is $\mathbb{R}^{4n}$: a point $P = (p_1, \dots, p_{4n})$ determines a \term{realization} of $G$ as a (possibly degenerate) graph in $\mathbb{R}^{4}$ with vertices $$\mathbf{p}_{i} = (p_{4i-3}, p_{4i-2}, p_{4i-1}, p_{4i}).$$ If the points $\mathbf{p}_{i}$ are spacelike and form the vertices of a de Sitter polyhedron, then $P$ provides a realization of $G$ as a circle polyhedron with the combinatorics of $G$ and vertex-disks $D_{i} = D_{\mathbf{p}_{i}}$. 

We want to study paths in the configuration space that fix inversive distances along each edge of a de Sitter polyhedron. Here, each vertex is associated with a spacelike vector and scaling the vector with a positive scalar has no effect on any of the inversive distances. Also, we want to use equation~\eqref{eq:inversivedistance} when computing the inversive distance, which requires that the vertices of our polyhedron be given by points on $\deSitter$. This motivates the following definition. 

Define the \term{inversive measure function} $$f:\mathbb{R}^{4n}\to\mathbb{R}^{4n-6}$$ from the configuration space to the \term{measure space} $\mathbb{R}^{4n-6}$ as follows. The first $3n-6$ coordinates of $\mathbb{R}^{4n-6}$ are indexed by the edges of $G$. The remaining $n$ coordinates are indexed by the vertices of $G$. For an edge $ij$, $f$ measures the Lorentz inner product of the edge: $$f_{ij}(P) =\langle \mathbf{p}_{i}, \mathbf{p}_{j}\rangle_{3,1}.$$ For a vertex $i$, $f$ measures the Lorentz inner product of $\mathbf{p}_{i}$ with itself: $$f_{i}(P) = \langle \mathbf{p}_{i},  \mathbf{p}_{i}\rangle_{3, 1}.$$

Suppose $P$ determines a circle polyhedron with the coordinates of each of its vertex-disks normalized to lie on $\deSitter$. Suppose that $P_t$ for $t\in[0, 1]$ is a continuous path in $\mathbb{R}^{4n}$ with $P_0 = P$. The construction of $f$ is designed so that any path in the configuration space that holds $f$ constant and contains one point corresponding to a de Sitter polyhedron is a continuous motion of de Sitter polyhedra that maintains all edge inversive distances. This means that no inversive distance changes on an edge through the motion {\em and} no vertex moves off the de Sitter sphere. The following lemma follows directly from the construction of $f$.

\begin{Lemma} Let $P$ determine a circle polyhedron with each of its vertex-disk coordinates scaled so that the coordinates are points on $\deSitter$. Let $P_t$ be a path in $\mathbb{R}^{4n}$ such that $P_0 = P$. Then $f(P_t)$ is constant for $t\in[0, 1]$ if and only if the following two statements hold:
\begin{enumerate}
    \item[\em{(i)}] for each edge $ij\in E(G)$ the inversive distance between $D_i$ and $D_j$ is constant for $t\in[0, 1]${\em ;}
    \item[\em{(ii)}] each vertex remains on $\deSitter$ for all $t\in [0, 1]$. 
\end{enumerate}
\end{Lemma}

An important property for the proof of our main theorem is that when $P$ determines a strictly convex triangulated circle polyhedron, it is a regular point of the inversive measure function.

\begin{Theorem}[\cite{bowers2020}]\label{thm:bowers2020}
Let $P$ determine a strictly convex triangulated c-polyhedron with $n$ vertices and $f:\mathbb{R}^{4n}\to\mathbb{R}^{4n-6}$ be the inversive measure function. Then the Jacobian $J$ of $f$ at $P$ has rank $4n-6$.  
\end{Theorem}

A consequence of this theorem is that a strictly convex circle polyhedron in $\mathbb{R}^{4n}$ can be approximated by a nearby strictly convex circle polyhedron in $\mathbb{R}^{4n}$ that realizes any small enough changes in inversive distance. 

\section{Properness of c-Polyhedra}\label{ProperDefinition}
We will use Theorem~\ref{thm:bbp18} in the proof of our main result. One of its hypotheses is that of the properness of the circle polyhedron, which was defined in \cite{BBP18} for non-unitary circle polyhedra. We need to generalize the definition of properness to circle polyhedra that allow for tangent circles, i.e., that fail to be non-unitary. At this juncture, we give a quick, working definition of properness after listing sufficient conditions that guarantee properness. We will return to a deeper study of properness in Section~\ref{ProperProofs}, where we will expand its definition and verify the sufficient conditions listed next.

\subsection{Sufficient conditions for properness}

\begin{Definition}
    A c-polyhedron with combinatorics $G=(V, E, F)$ is
    \begin{itemize}
        \item 
            \term{hyperideal} if all its vertex-disks are pairwise disjoint;
        \item
            \term{locally shallow} if $\inv(D_v, D_w) > 0$ for all adjacent $v, w\in V$,
            and $D_v$ and $D_w$ are disjoint if $\inv(D_v, D_w)\geq 1$;
        \item
            \term{globally shallow} if $\inv(D_v, D_w) > 0$ for all $v, w\in V$, $v\neq w$,
            and $D_v$ and $D_w$ are disjoint if $\inv(D_v, D_w)\geq 1$;
        \item
            a \term{Koebe polyhedron} if adjacent vertex-disks touch externally and all other
            pairs of vertex-disks are disjoint.
    \end{itemize}
\end{Definition}

It is immediate that a c-polyhedron is globally shallow if it is hyperideal or a Koebe polyhedron. The following theorem, whose proof is delayed until Section~\ref{ProperProofs}, shows that these classes of polyhedra are also proper. All the polyhedra considered in KAT, \cite{BaoBonahon:2002}, and \cite{BBP18} are globally shallow. 

\begin{Theorem}\label{lm:propershallow}
    A convex hyperbolic c-polyhedron is proper whenever it is globally shallow.
\end{Theorem}

\subsection{A working definition of properness}
Let ${P}$ be a hyperbolic circle polyhedron with combinatorics $(V, E, F)$ and $v\in V$ a vertex with vertex-disk $D_{v}$. Give to $D_{v}$ the Poincar\'e metric of constant curvature $-1$ making $D_{v}$ a model of the hyperbolic plane. Let $f_{1}, \dots, f_{n}$ be the faces of $F$ that contain the vertex $v$ and ordered cyclically around $v$, and let $O_{1},\dots, O_{n}$ be the corresponding orthocircles determined by the faces and oriented oppositely to the orientation on $\mathbb{S}^2$. Each $O_{i}$ intesects $D_{v}$ in a directed hyperbolic line $\ell_{i} = O_{i} \cap D_{v}$, an arc of the circle $O_{i}$ that meets the boundary circle $\partial D_{v}$ orthogonally. The lines $\ell_{i}$ bound half-planes to their left whose intersection $H_{v}$ is a convex subset of $D_{v}$. $H_{v}$ is called a \term{hyperideal} hyperbolic polygon and need not be compact, nor even of finite area. 

We identify three types of possible vertices for $H_{v}$. When $\ell_{i}$ meets $\ell_{i+1}$ at the point $v_{i}$ in $D_{v}$, we call $v_{i}$
a \term{visible vertex}. When $\ell_{i}$ meets $\ell_{i+1}$ at the ideal point $v_{i}$ on the boundary $\partial D_{v}$, we call $v_{i}$ an \term{ideal vertex}. Finally, when $\ell_{i}$ and $\ell_{i+1}$ have empty intersection, there is a unique hyperbolic line segment $v_{i}$ orthogonal to both and we call $v_{i}$ a \term{hyperideal vertex}. (These definitions will be refined in Section~\ref{ProperProofs}.)

The hyperideal polygon $H_{v}$ need not contain all the vertices $v_{i}$ determined by pairs $\ell_{i}$ and $\ell_{i+1}$ of consecutive lines. It may be that there is a line $\ell_{j}$ that cuts $v_{i}$ off from $H_{v}$, or even one that meets the hyperideal vertex $v_{i}$ in its open segment. $H_{v}$ is said to be \term{proper} if each $\ell_{i}$ meets $H_{v}$ in a subarc of non-zero length and all the vertices $v_{i}$ determined by consecutive lines lie in $H_{v}$. If $H_{v}$ is proper, let $S_{i}$ be the hyperbolic line that contains the hyperideal vertex $v_{i}$ and oriented oppositely to that of $\mathbb{S}^{2}$. The \textit{link} of the vertex-disk $D_{v}$ is the boundary $L_{v}$ of the finite-area hyperbolic polygon obtained by intersecting $H_{v}$ with the half-planes to the left of the $S_{i}$'s. 

The circle polyhedron $P$ is \term{proper} if $H_{v}$ is proper for every vertex $v$. In this case the link of each vertex-disk bounds a finite-area convex hyperbolic polygon. It turns out to be rather difficult to construct improper circle polyhedra, and Lemma~\ref{lm:propershallow} shows that any improper circle polyhedron must have \term{deep overlaps}, meaning that there must be vertex-disks that overlap by greater than $\pi/2$ radians.

This presents a quick introduction to properness, but the details will not be used in the proof of our main result. For this reason, we proceed to the main result and delay a more intense discussion of properness until Section~\ref{ProperProofs}.

\section{Uniqueness of Triangulated Unitary c-Polyhedra}\label{UniqueTriangulated}
Properness for a non-unitary circle polyhedron guarantees that the links of vertex-disks are compact, and this allows one to prove a Cauchy arm lemma in the setting of generalized polygonal paths in the hyperbolic plane with both visible and hyperideal vertices; see Section~3.4 of~\cite{BBP18}. When the circle polyhedron fails to be non-unitary, the links may have ideal vertices. When this happens, the Cauchy arm lemma fails because there are infinite-length edges in the link. This is precisely where the non-unitary hypothesis is needed in Theorem~\ref{thm:bbp18} for the proof in~\cite{BBP18}. In this section, we prove our main result, which removes the non-unitary hypothesis from Theorem~\ref{thm:bbp18} for triangulated circle polyhedra and allows for vertex-disks to be tangent, and thus for ideal vertices in links. 

\begin{Theorem}\label{thm:triangulatedcase}
    Let $P$ and $Q$ be locally M\"obius congruent proper strictly convex triangulated hyperbolic circle polyhedra. Then $P$ and $Q$ are M\"obius congruent.
\end{Theorem}

\begin{proof}
    Suppose that $P$ and $Q$ are locally M\"obius congruent proper strictly convex triangulated c-polyhedra with $n$ vertices and $m=3n-6$ edges with combinatorics $G$. If no edge has unit inversive distance, then Theorem~\ref{thm:bbp18} already shows that $P$ and $Q$ are congruent. Assume then at least one edge has a unit inversive distance.

    \paragraph{Sketch of Proof.} We first establish the existence of two continuous families of circle polyhedra $P_t$ and $Q_t$, defined for $t\in [0, 1]$, where $P_0 = P$ and $Q_0 = Q$, $P_t$ is locally congruent to $Q_t$ for all $t\in[0, 1]$, and for every $t\in (0, 1]$, no inversive distance is equal to 1. We then apply Theorem~\ref{thm:bbp18} to produce a family of M\"obius transformations $\phi_t$ that map $P_t$ onto $Q_t$ for all $t\in(0, 1]$. Finally, we show that despite the non-compactness of the M\"obius group, $\phi_t$ converges to a M\"obius transformation $\phi$ as $t\to 0$, and $\phi(P)=Q$, which establishes our theorem. 

\begin{figure}[t!h]
	\centering
	\includegraphics[width=0.65\textwidth]{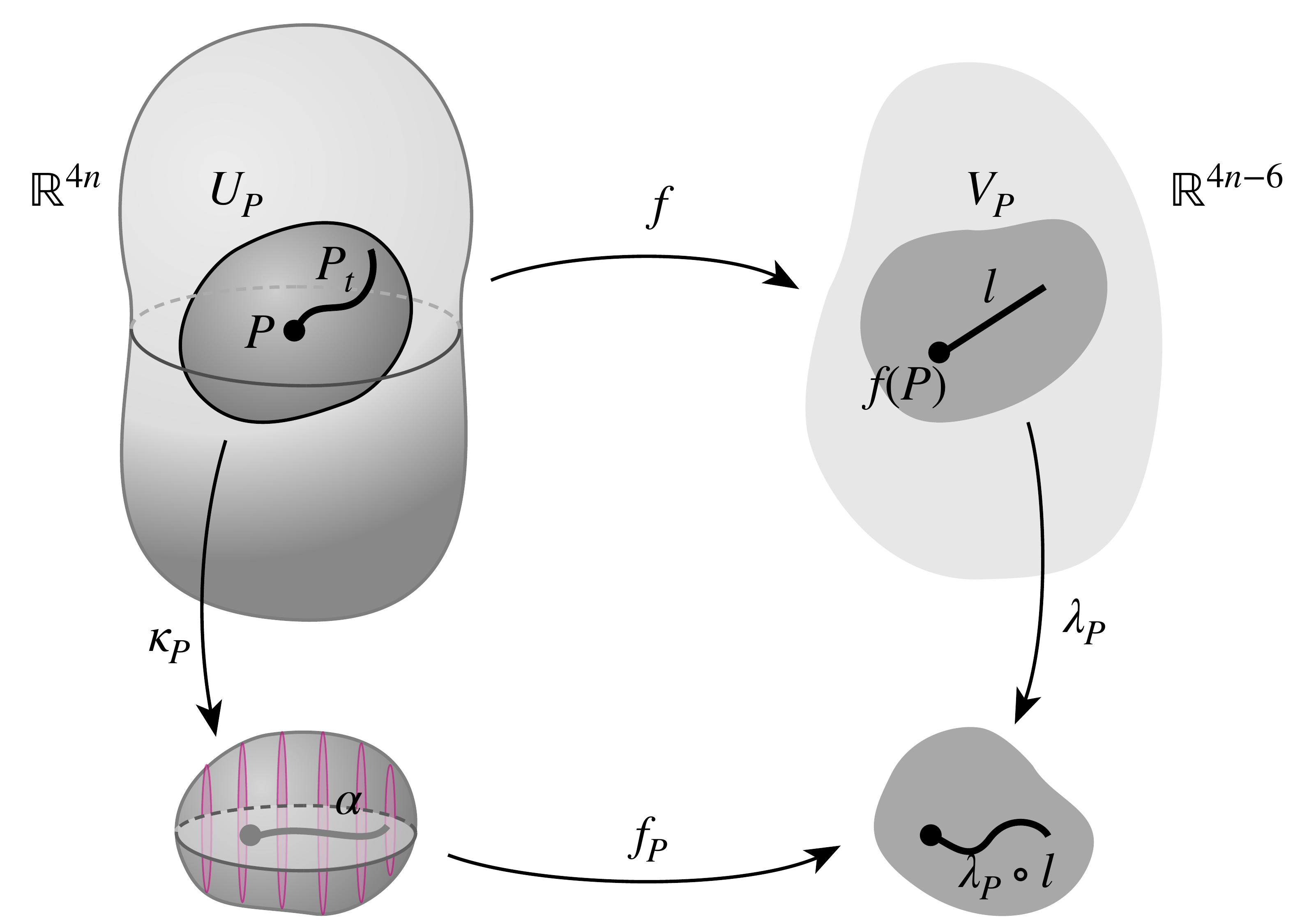}
	\caption{The construction of $P_t$. (The construction of $Q_t$ is similar.)}
	\label{fig:contructpt}
\end{figure}

\paragraph{The non-unitary polyhedral families $P_t$ and $Q_t$ that limit to $P$ and $Q$.} Consider $P$ and $Q$ as points in the configuration space $\mathbb{R}^{4n}$ and let $f:\mathbb{R}^{4n}\rightarrow\mathbb{R}^{4n-6}$ be the inversive measure function. Since properness and strict convexity are open conditions, there exist open balls $B_P$ centered at $P$ and $B_Q$ centered at $Q$ such that every $P'\in B_P$ and every $Q'\in B_Q$ is strictly convex and proper. A precise statement and proof of this fact appears as Lemma~\ref{lem:convexityopen} in Section~\ref{ProperProofs} after a more thorough development of properness in that section. By Theorem~\ref{thm:bowers2020}, the rank of the Jacobian $J$ of $f$ on the open sets $B_{P}$ and $B_{Q}$ is constantly $4n-6$. 

    Since $J$ has constant full rank on the open neighborhood $B_{P}$, the constant rank theorem says that there are coordinate charts, $(U_{P}\subset B_{P}, \kappa_{P})$ at $P$ and $(V_{P}, \lambda_{P})$ at $f(P)$ with $f(U_{P}) \subset V_{P}$, such that the local expression for $f$ is projection to the first $4n-6$ coordinates.; i.e., the local mapping $f_{P} = \lambda_{P} \circ f \circ \kappa_{P}^{-1}$ takes the form
    $$f_{P} :(x_{1}, \dots ,x_{4n}) \mapsto (x_{1} \dots ,x_{4n-6})$$
    for every $(x_{1},\dots, x_{4n}) \in \kappa_{P}(U_{P})$. The fibers of $f_{P}$ correspond to M\"obius transformations and are the 6-dimensional linear spaces coordinatized by the last 6 coordinates of $\kappa_{P}(U_{P})$. The charts $(U_{Q}\subset B_{Q}, \kappa_{Q})$ at $Q$ and $(V_{Q}, \lambda_{Q})$ at $f(Q)$ are defined similarly, and the local mapping $f_{Q} = \lambda_{Q} \circ f \circ \kappa_{Q}^{-1}$ takes the form
    $$f_{Q} :(y_{1}, \dots ,y_{4n}) \mapsto (y_{1} \dots ,y_{4n-6})$$
    for every $(y_{1},\dots, y_{4n}) \in \kappa_{Q}(U_{Q})$.

    Since $P$ and $Q$ are locally congruent, $f(P) = f(Q)$. Let $V = f(U_P)\cap f(U_Q)$. Since both $U_P$ and $U_Q$ are open and $f$ is smooth on $B_P$ and $B_Q$, $V$ is an open neighborhood of $f(P)$. 

    Every edge in $P$ corresponds to a coordinate axis in the measure space $\mathbb{R}^{4n-6}$. Let $\vec{e}_1, \dots \vec{e}_k$ denote the basis vectors for the coordinate axes corresponding to the edge measures of $P$ with unit inversive distance (i.e., the edges $ij$ where $f_{ij}(P) = 1$). Let $\vec{v}=\sum_{i} \vec{e}_i$ and let $l(t) = f(P)+\mu t \vec{v}$ for $t\in[0, 1]$ and a sufficiently small $\mu > 0$ such that the entire line segment $l([0,1])\subset V$. 
 \vskip6pt   
\noindent\textbf{Remark:} $l(t)$ represents a path of inversive distance measurements for $G$ where only the inversive distances equal to 1 in $l(0)$ are increasing and moving away from unitary, while all other inversive distances remain constant throughout the path. In particular, there are no unit inversive distances represented by $l(t)$ for $t>0$.
\vskip6pt
     Let $\kappa_P(P) = (p_1, \dots, p_{4n})$ and $\kappa_Q(Q) = (q_1, \dots, q_{4n})$. Let $(\alpha_1(t), \dots, \alpha_{4n-6}(t))$ denote the coordinates of $\lambda_{P}\circ l(t)$. Then $$\alpha(t) = (\alpha_1(t), \dots, \alpha_{4n-6}(t), p_{4n-5}, \dots, p_{4n})$$ gives a curve in $\kappa_P(U_P)$ that projects onto $\lambda_{P} \circ l$ via $f_{P}$. Similarly, $$\beta(t) = (\beta_1(t), \dots, \beta_{4n-6}(t), q_{4n-5}, \dots, q_{4n})$$ gives a curve in $\kappa_Q(U_Q)$ that projects onto $\lambda_{Q}\circ l (t) = (\beta_{1}(t), \dots, \beta_{4n-6}(t))$ via $f_{Q}$.

    We now pull these two curves back to $U_P$ and $U_Q$, respectively, to obtain continuous paths $P_t = \kappa_P^{-1}(\alpha(t))$ and $Q_{t} =\kappa_Q^{-1}(\beta(t))$. By construction, $P_0 = P$, $Q_0 = Q$ and for all $t\in[0, 1]$, $f(P_{t}) =l(t) = f(Q_{t})$. Moreover, for all $t\in(0, 1]$, no edge $ij$ in $P_t$ (resp.,~$Q_t$) has $f_{ij}(P_t) = 1$. By Theorem~\ref{thm:bbp18} it follows that, for each $t\in(0, 1]$, there exists a M\"obius transformation $\phi_t$ that maps $P_t$ to $Q_t$. 

    \paragraph{Mapping $P$ to $Q$.}
    It remains to show that there exists a M\"obius transformation $\phi$ such that $\phi(P)=Q$. Choose any face of $G$ and label its vertices $1$, $2$, and $3$. Let $\mathbf{p}_1$, $\mathbf{p}_2$, and $\mathbf{p}_3$ denote the respective de Sitter points corresponding to vertices $1$, $2$, and $3$ in $P$, and similarly for $\mathbf{q}_1$, $\mathbf{q}_2$, and $\mathbf{q}_3$ in $Q$. Assume that $0$ is not the name of any vertex of $G$ and let $\mathbf{p}_{0}$ be the unit Lorentz normal to the hyperplane spanned by $\mathbf{p}_{1}$, $\mathbf{p}_{2}$, and $\mathbf{p}_{3}$ that has positive last coordinate, and similarly for $\mathbf{q}_{0}$. Let $\phi$ denote the unique M\"obius transformation that sends $\mathbf{p}_i\mapsto \mathbf{q}_i$ for $i=1,2,3$ and note that $\phi(\mathbf{p}_{0}) = \mathbf{q}_{0}$. We will show that $\phi(P) = Q$; i.e., we will show that for each vertex $\ell$ of $G$, $\phi(\mathbf{p}_{\ell}) = \mathbf{q}_{\ell}$ where, of course, $\mathbf{p}_{\ell}$ is the de Sitter point in $P$ corresponding to $\ell$, and similarly for $\mathbf{q}_{\ell}$ in $Q$. We will use the following facts:
    \begin{enumerate}
     \item[(i)] The collection $\mathcal{B} = \{\mathbf{q}_{0}, \mathbf{q}_{1}, \mathbf{q}_{2}, \mathbf{q}_{3}\}$ is a vector space basis for $\mathbb{R}^{3,1}$.
    \item[(ii)] The Lorentz inner product is non-degenerate, which easily implies that any element $\mathbf{u}$ of $\mathbb{R}^{3,1}$ is uniquely determined by its Lorentz inner product with any basis. 
    \item[(iii)] The Lorentz inner product is continuous in its arguments.
    \end{enumerate}
    
\noindent Let $\ell$ be any vertex of $G$. Our aim is to verify that for each basis element $\mathbf{q}_{i} \in \mathcal{B}$, 
\begin{equation}\label{eq:PQ}
    \langle \mathbf{q}_{i} , \phi(\mathbf{p}_{\ell} ) \rangle_{3,1} = \langle \mathbf{q}_{i} ,\mathbf{q}_{\ell} \rangle_{3,1}.
\end{equation}
By item (ii) above, we may then conclude that $\phi(\mathbf{p}_{\ell}) = \mathbf{q}_{\ell}$ and hence $\phi(P) = Q$.

To verify equation~\eqref{eq:PQ}, for any vertex $v$ of $G$ and $t\in [0,1]$, let $\mathbf{p}_{v}(t)$ and $\mathbf{q}_{v}(t)$ be the de Sitter points in $P_{t}$ and $Q_{t}$, respectively, that correspond to $v$. Note that for each vertex $v$, $\mathbf{p}_{v}(t) \rightarrow \mathbf{p}_{v}(0) = \mathbf{p}_{v}$ and $\mathbf{q}_{v}(t) \rightarrow \mathbf{q}_{v}(0) = \mathbf{q}_{v}$ as $t\rightarrow 0+$. Also, letting $\mathbf{p}_{0}(t)$ be the unit Lorentz normal to the hyperplane spanned by $\mathbf{p}_{1}(t)$, $\mathbf{p}_{2}(t)$, and $\mathbf{p}_{3}(t)$, and similarly for $\mathbf{q}_{0}(t)$, we have $\mathbf{p}_{0}(t) \rightarrow \mathbf{p}_{0}(0)= \mathbf{p}_{0}$ and $\mathbf{q}_{0}(t) \rightarrow \mathbf{q}_{0}(0) = \mathbf{q}_{0}$ as $t\rightarrow 0 +$. By item (iii) above we get, for each $i = 0,1,2,3$,
\begin{equation}\label{eq:pq}
\langle \mathbf{p}_{i} , \mathbf{p}_{\ell}\rangle_{3,1} = \lim_{t\to 0+} \langle \mathbf{p}_{i}(t), \mathbf{p}_{\ell}(t) \rangle_{3,1} \quad\textrm{and}\quad \langle \mathbf{q}_{i} , \mathbf{q}_{\ell}\rangle_{3,1} = \lim_{t\to 0+} \langle \mathbf{q}_{i}(t), \mathbf{q}_{\ell}(t)\rangle_{3,1}.
\end{equation}
Since $\phi_{t}$ maps $P_{t}$ to $Q_{t}$, we have $\phi_{t}(\mathbf{p}_{i}(t))= \mathbf{q}_{i}(t)$ for $i = 0,1,2,3,\ell$, and since $\phi_{t}$ is M\"obius and therefore a Lorentz transformation that preserves the Lorentz inner product, we have $$\langle \mathbf{p}_{i}(t), \mathbf{p}_{\ell}(t) \rangle_{3,1} = \langle \phi_{t}(\mathbf{p}_{i}(t)), \phi_{t}(\mathbf{p}_{\ell}(t) )\rangle_{3,1}  = \langle \mathbf{q}_{i}(t), \mathbf{q}_{\ell}(t)\rangle_{3,1}.$$  This with equations~\eqref{eq:pq} implies that $\langle \mathbf{p}_{i} , \mathbf{p}_{\ell}\rangle_{3,1}  = \langle \mathbf{q}_{i} , \mathbf{q}_{\ell}\rangle_{3,1}$. Finally observe that since $\phi(\mathbf{p}_{i}) = \mathbf{q}_{i}$ for $i=0,1,2,3$ and $\phi$ is M\"obius that preserves the Lorentz inner product, 
$$\langle \mathbf{q}_{i} , \phi(\mathbf{p}_{\ell} ) \rangle_{3,1} =  \langle \phi(\mathbf{p}_{i}) , \phi(\mathbf{p}_{\ell} ) \rangle_{3,1} =  \langle \mathbf{p}_{i} , \mathbf{p}_{\ell}\rangle_{3,1} = \langle \mathbf{q}_{i} , \mathbf{q}_{\ell}\rangle_{3,1},$$ verifying equation~\eqref{eq:PQ}. Hence $\phi(\mathbf{p}_{\ell}) = \mathbf{q}_{\ell}$ for every vertex $\ell$ of $G$ and so $\phi(P) = Q$.
\end{proof}

Theorem~\ref{TheoremA} is a restatement of Theorem~\ref{thm:triangulatedcase} in terms of hyperideal polyhedra and Theorem~\ref{TheoremB} is a restatement of Theorem~\ref{thm:triangulatedcase} in terms of inversive distance circle packings.

\section{Properness Revisited}\label{ProperProofs}
We refer the reader to Sections~1.4--1.6 of~\cite{BBP18} for background and details of properness for non-unitary circle polyhedra. In this section, we revisit the definition, refine it, and explore conditions on a circle polyhedron that will imply properness.

\subsection{Hyperbolic polygons}

We want to find conditions for c-polyhedra to guarantee that the links of its vertices
are well-behaved. We begin by considering convex polygons in the hyperbolic plane
$\HH^2$. To a finite number of points $\{p_0, \dotsc, p_{n-1}\}\subset\HH^2$ we can
naturally associate a convex polygon via their \term{convex hull}
$\conv(p_0, \dotsc, p_{n-1})$. As in the Euclidean plane, it can be defined as the
smallest convex set containing all $p_i$. Equivalently, it is the intersection
\begin{equation}\label{eq:convex_hull}
    \conv(p_0, \dotsc, p_{n-1})
    =
    \bigcap_{\substack{H \text{ half-plane} \\ \forall i: p_i \in H}} H
\end{equation}
of all hyperbolic half-planes containing $\{p_0, \dotsc, p_{n-1}\}$. The convex hull
$\conv(p_0, \dotsc, p_{n-1})$ is a compact convex subset of the hyperbolic plane.
It is bounded by a finite number of hyperbolic segments whose endpoints are a subset
of the initial set of points $\{p_0, \dotsc, p_{n-1}\}$---the \term{vertices} of
the polygon. This follows directly from the fact that hyperbolic geodesics are modeled
in the Betrami--Klein model by Euclidean straight lines intersecting the unit disk centered
at the origin. This observation also shows that we can admit ideal points
$p_i\in\partial\HH^2$, i.e., points on the `ideal` circle bounding the Beltrami--Klein
model, in the defining equation \eqref{eq:convex_hull} of the convex
hull. In this case, the convex hull is still convex and has finite area, but it is not
bounded and thus no longer compact.

Consider now the converse construction:\;given a finite number $n\geq3$ of half-planes
$H_0, \dotsc, H_{n-1}$ in the hyperbolic plane, call the intersection
\begin{equation}
    \poly(H_0, \dotsc, H_{n-1})
    =
    \bigcap_{i=0}^{n-1}H_i
\end{equation}
the \term{polygon} of $H_0, \dotsc, H_{n-1}$. The collection of half-planes is \term{minimal}
if we cannot remove any of them without changing their intersection (as a set). Assume
that $H_0, \dotsc, H_{n-1}$ is minimal. Since hyperbolic half-planes are in one-to-one
correspondence to segments of the unit circle, we can assume that (possibly after
relabeling) the order of $H_0, \dotsc, H_{n-1}$ is compatible with the counter-clockwise
orientation of the circle. More precisely, each half-plane is bounded by a hyperbolic geodesic
that is adjacent to exactly two ideal vertices $p^+, p^-\in\partial\HH^2$. Choose the geodesic orientation in such a way that it starts at $p^+$, ends at $p^-$, and the half
plane lies to its left. We will denote by $H(p^+, p^-)$ the half-plane determined by $p^{\pm}$.
Up to multiples of $2\pi$, the segment $(p^+, p^-)\subset\partial\HH^2$ can be identified with an
interval $(\varphi^+, \varphi^-)\subset\RR$ via $\mathbb{R} \ni \varphi\mapsto\ e^{i\varphi} = p \in \partial\mathbb{H}^{2}$. Our
choice of orientation implies $\varphi^+<\varphi^-$. Hence, we can order the half-planes
$H_0, \dotsc, H_{n-1}$ such that $\varphi_i^{\pm}<\varphi_{i+1}^{\pm}$.
Here, $\varphi_{i+ns} = \varphi_i+2\pi s$.
Thus, the geodesics bounding two consecutive half-planes $H_{i-1}$ and $H_{i}$ (all indices
modulo $n$) either
\begin{enumerate*}[label={(\roman*)}, itemjoin={;\ }]
    \item
        intersect in a point $v_i\in\HH^2$
    \item
        intersect in an ideal point $v_i=\ee^{\varphi_{i-1}^-}=\ee^{\varphi_{i}^+}$
    \item
        do not intersect and thus have a common perpendicular geodesic with endpoints
        $v_i^+, v_i^-\in\partial\HH^2$. We label these points such that
        $v_i^+\in (p_{i-1}^+, p_{i-1}^-)$
        and $v_i^-\in(p_i^+, p_i^-)$.
\end{enumerate*}
In the first case we call $v_i$ a \term{visible vertex} and in the second an
\term{ideal vertex}. In the third case $v_i = (v_i^+, v_i^-)$ is a
\term{hyperideal vertex}.
The \term{truncation} of the polygon of $H_0, \dotsc, H_{n-1}$ is
\begin{equation}
    \trunc(H_0, \dotsc, H_{n-1})
    =
    \poly(H_0, \dotsc, H_{n-1}) \cap \bigcap_{i: v_i \text{ hyperideal}} H(v_i^+, v_i^-).
\end{equation}
The truncation is always a polygon that can be obtained as the convex hull of a finite number
of points in $\HH^2\cup\partial\HH^2$. We call the polygon $\poly(H_0, \dotsc, H_{n-1})$
\term{proper} if its visible and ideal vertices are also vertices of the truncation.

Proper polygons play a special role as they behave similarly to convex compact polygons.
In particular, we can define edge-lengths and angles for these polygons. For the segments
of the truncation $\trunc(H_0, \dotsc, H_{n-1})$ of a proper polygon there are two cases:\;either
\begin{enumerate*}
    \item[(i)] the segment is supported by the boundary geodesic of a
        half-plane $H_i$, or
    \item[(ii)]
        it is supported by the boundary geodesic of a hyperideal vertex $v_j$.
\end{enumerate*}
In the former case we call the segment a \term{black edge} and in the latter a
\term{green edge}. This recovers the black-green polygons defined in \cite{BBP18}.
We call the hyperbolic lengths of the black edges the \term{edge lengths} of the
polygon $\poly(H_0, \dotsc, H_{n-1})$. Furthermore, we can define an angle at
each vertex. If $v_i$ is visible, then it is the interior angle between the two
black edges adjacent to $v_i$. If $v_i$ is hyperideal, then it has a corresponding
green edge in the truncation. We define its angle to be the length of this green edge times the imaginary unit.
It is a classical fact that proper triangles satisfy a generalized law of cosines
(see, e.g., \cite{Thurston1997}). We will only need the following special case

\begin{Lemma}\label{lemma:law_of_cosines}
    Consider a proper triangle with two visible vertices $A, B$ and one hyperideal
    vertex $C$. Let $\alpha\in(0, \pi)$ be the angle at $A$ and $a, b, c\in\RR_{>0}$ the
    edge length of the black edges opposite to $A, B, C$, respectively. Then
    \begin{equation}
        \cos(\alpha) = \frac{\sinh(b)\cosh(c) - \sinh(a)}{\cosh(b)\sinh(c)}.
    \end{equation}
\end{Lemma}

Let $G = (V, E, F)$ be the combinatorics of the convex hyperbolic circle polyhedron $P$. Denote by $(D_v)_{v\in V}, (D_f)_{f\in F}\subset\mathbb{S}^{2}$ the
disks associated to the vertices and faces, respectively. Let $v\in V$ and let
$v_0, v_1, \dotsc, v_{n-1}$ be the vertices in the link of $v$ labeled in
counter-clockwise order. Furthermore, denote the orthodisk of the face
$\{v, v_i, v_{i+1}\}$ by $D_i^*$. Interpreting the disk $D_v$ as a hyperbolic plane
(in the Poincaré disk model), the $D_v\setminus D_i^*$ give hyperbolic half-planes. Thus,
by a slight abuse of notation, we can associate the hyperbolic polygon
\begin{equation}
    \poly(D_0^*, \dotsc, D_{n-1}^*)
\end{equation}
to the vertex $v$. We call $v$ \term{proper} if its associated polygon is proper and we
say that the c-polyhedron $P$ is proper if all its vertices are.

Let us give another geometric interpretation of properness not depending on hyperbolic
geometry. Consider the link of a vertex $v\in V$ as before. A visible or ideal vertex
$p_i$ of the polygon associated to $v$ corresponds to a disk $D_{v_i}$ intersecting
$\overline{D}_v$ in at most one point. It is the unique point
$\partial D_{i-1}^*\cap\partial D_i^*\cap\overline{D}_v$. Equivalently, it is the unique
point in the pencil of circles $D_{v_i}\wedge D_v$ spanned by $D_{v_i}$ and $D_v$ lying
inside $\overline{D}_v$. This point is given by $(\lambda-1)D_{v_i} - \lambda D_v$
where $\lambda\in\RR$ is the unique solution to the quadratic equation
\begin{equation}\label{eq:points-in-pencil}
    0
    =
    \|(\lambda-1)D_{v_i} - \lambda D_v\|_{3,1}^2,
\end{equation}
satisfying $\ip{(\lambda-1)D_{v_i}-\lambda D_v}{D_v}\geq0$.
In turn, a hyperideal vertex corresponds to a disk $D_{v_j}$ intersecting $D_v$.
The hyperbolic half-plane associated to the hyperideal vertex is given by the
unique disk $D_{\mu} = (\mu-1)D_{v_j}-\mu D_v$ in the pencil
$D_{v_j}\wedge D_{v}$ where $0 = \ip{(\mu-1)D_{v_j}-\mu D_v}{D_v}_{3,1}$, which is given by the parameter value
\begin{equation}\label{eq:ortho-disk-in-pencil}
    \mu =  \frac{\ip{D_{v_j}}{D_{v}}_{3,1}}{\ip{D_{v_j}}{D_v}_{3,1}-1}.
\end{equation}
The properness of $\poly(D_0^*, \dotsc, D_{s-1}^*)$ says that no visible or
ideal vertex is contained in this disk $D_{\mu}$ defining the hyperideal vertex $p_j$.
Equivalently, letting $\hat{p}_{j}$ be the de Sitter point corresponding to the disk $D_{\mu}$ of~\ref{eq:ortho-disk-in-pencil}, 
\begin{equation}\label{eq:properness-condition}
    \ip{p_i}{\hat{p}_j}_{3,1} \leq 0
\end{equation}
holds for all visible or ideal vertices $p_i$ and all hyperideal vertices $p_j$.
The two solutions to \eqref{eq:points-in-pencil} are related through inversion in
$\partial D_v$. Since the disk defining a hyperideal vertex is orthogonal to
$\partial D_v$ it is preserved under such inversions. This leads to the following
alternative characterization of properness.

\begin{Lemma}\label{lemma:properness-by-pencils}
    Let $P$ be a convex hyperbolic c-polyhedron. To each ordered pair of adjacent
    vertices $(v,w)$ we can associate
    \begin{itemize}
        \item[{\em (i)}]
            a point in $\SS^2$ corresponding to the solution of \eqref{eq:points-in-pencil} with $w=v_{i}$
            if the vertex-disks $D_{v}$ and $D_{w}$ intersect in at most a point;
        \item[{\em (ii)}]
            a disk in $\SS^2$ given as the solution to \eqref{eq:ortho-disk-in-pencil} with $w=v_{j}$
            if the vertex-disks $D_{v}$ and $D_{w}$ intersect in more than one point.
    \end{itemize}
    The polyhedron $P$ is proper if and only if none of these points in {\em (i)} is contained in
    any of the disks in {\em (ii)}.
\end{Lemma}

\begin{Lemma}\label{lem:convexityopen}
    Let $P$ be a strictly convex hyperbolic proper circle polyhedron. Then there is an open ball $B\subset \mathbb{R}^{4n}$ centered at $P$ in the configuration space of $P$ such that all circle polyhedra $Q\in B$ are strictly convex, hyperbolic, and proper. 
\end{Lemma}

\begin{proof}

{\bf Strict convexity is an open condition in $\mathbb{R}^{4n}$:}
Let $P$ be a strictly convex proper circle polyhedron. Let $(a_i, b_i, c_i, d_i)$ denote the coordinates of the disk corresponding to vertex $i$ of $P$. Define the following edge determinant function $\Psi:E(P)\to\mathbb{R}$ by $$\Psi(ij)=\begin{vmatrix} a_i & b_i & c_i & d_i \\ a_j & b_j & c_j & d_j \\ a_k & b_k & c_k & d_k \\ a_l & b_l & c_l & d_l\end{vmatrix}$$ where $ijk$ and $jil$ are the two triangles incident to $ij$ in $G$. $P$ is strictly convex if and only if none of its edge determinants are zero and all of its edge determinants have the same sign (see \cite{bowers2020} for a full development of this). There is an open ball centered at $P$ in $\mathbb{R}^{4n}$ where no edge determinant equals zero.  

{\bf Hyperbolicity is an open condition in $\mathbb{R}^{4n}$:}
Hyperbolicity says that we can find to each face $f$ of $P$ an orthodisk $D_f$. This
disk is the solution to the system of equations
\begin{align}
    1 &= \|D_f\|_{3,1}\\
    0 &= \ip{D_f}{D_v} &\text{if } v<f,\\
    0 &< \ip{D_f}{D_v} &\text{otherwise,} 
\end{align}
which depends continuously on the vertex-disks of $P$.

{\bf Properness is an open condition in $\mathbb{R}^{4n}$:}
This follows from lemma \ref{lemma:properness-by-pencils} and, more precisely, from the
continuous dependence of the solutions of \eqref{eq:points-in-pencil} and
\eqref{eq:ortho-disk-in-pencil}, and hence of the properness condition
\eqref{eq:properness-condition}, on the vertex-disks of $P$. The only minor difficulty
are touching vertex-disks since generic variations will introduce another intersection
point or make the vertex-disks disjoint. The touching point can be computed either
from \eqref{eq:points-in-pencil} or \eqref{eq:ortho-disk-in-pencil}. Depending
on the variation, one of these equations remains valid and its solutions remain in a
neighborhood of the touching point.
\end{proof}

\subsection{Proof of Theorem~\ref{lm:propershallow}}

\begin{proof}
    Consider a vertex $v\in V$. In light of lemma \ref{lemma:properness-by-pencils},
    we only have to show that if $u, w\in V$ are adjacent to $v$ with
    $D_u\cap D_v=\emptyset$ and $D_w\cap D_v\neq\emptyset$, then the points
    in the pencil $D_u\wedge D_v$ are not contained in the disk $D_{\mu}$ from
    the pencil $D_w\wedge D_v$ determined by equation \eqref{eq:ortho-disk-in-pencil}.
    It is enough to argue for one of the points in $D_u\wedge D_v$ since they are
    related by inversion in $D_v$.
    
    Interpret $D_v' = \SS^2\setminus D_v$ as the Poincaré disk model of the hyperbolic
    plane and denote by $p$ the point in $D_u\wedge D_v$ contained in $D_v'$.
    Then $D_u\subset D_v'$ is a hyperbolic circle with center $p$ with respect to
    the Poincaré metric on $D_v'$. Denote its hyperbolic radius by $r_u>0$.
    Furthermore, $L = \partial D_{\mu}\cap D_v'$ is a hyperbolic geodesic. Any point in
    $D_v'$ has a well-defined signed distance to this geodesic, which we define to be
    positive if the point lies in $D_v'\setminus D_{\mu}$ and negative otherwise.
    Since $D_w$ and $D_{\mu}$ intersect $\partial D_v'$ in the same points (by construction),
    $\partial D_w$ has constant distance $r_w$ to the geodesic. The shallowness assumption
    implies that $\partial D_w\subset D_v'\setminus D_{\mu}$. Thus $r_w>0$.

    Suppose that $D_u$ intersects $D_w$ at an angle $\alpha$. Shallowness implies
    $\alpha\in(0, \frac{\pi}{2})$. Let $q$ be one of the intersection points of
    $\partial D_u$ and $\partial D_w$. The geodesic segments realizing the distances
    $r_u$ and $r_w$ between $q$ and $p$ or $q$ and $L$, respectively, meet at the angle
    $\pi - \alpha$. Hence, lemma \ref{lemma:law_of_cosines} shows that the distance
    $d\in\RR$ between $p$ and $L$ satisfies
    \begin{align}
        \sinh(d)
        &= 
        \sinh r_w \cosh r_u - \cos(\pi - \alpha)\cosh r_w \sinh r_u\\
        &=
        \sinh r_w \cosh r_u + \cos(\alpha)\cosh r_w \sinh r_u\\
        &>
        \sinh r_w \cosh r_u\\
        &>
        0.
        \qedhere
    \end{align}
\end{proof}

\subsection{Conjectures}
We close this article with a couple of conjectures.
\vskip6pt
\noindent\textbf{Conjecture.} A hyperbolic, locally shallow, convex c-polyhedron is globally shallow.
\vskip6pt
\noindent\textbf{Conjecture.} A convex hyperbolic c-polyhedron has a convex hyperbolic dual if and only if it has an orientation of its vertex-disks so that all inversive distances are in $[0, \infty)$. 
\vskip12pt
\noindent\textbf{Acknowledgements.}\;We thank Stephen P. Gortler for many helpful conversations related to the content of this paper.

This material is based upon work supported by the National Science Foundation under Grant No.\,DMS-1929284 while the authors were in residence at the Institute for Computational and Experimental Research in Mathematics in Providence, RI, during the \textit{Geometry of Materials, Packings and Rigid Frameworks} semester program. The first author was supported by a James Madison University Faculty Educational Leave.

\bibliography{OurBib}{}
\bibliographystyle{plain}

\appendix

\end{document}